\documentclass{amsart}%
\usepackage{amsfonts}
\usepackage{amsmath}
\usepackage{amssymb}
\usepackage{graphicx}%
\newtheorem{theorem}{Theorem}
\theoremstyle{plain}

\newtheorem{corollary}{Corollary}

\newtheorem{definition}{Definition}

\newtheorem{lemma}{Lemma}

\newtheorem{proposition}{Proposition}
\newtheorem{remark}{Remark}

\numberwithin{equation}{section}
\begin{document}
\title{On a remarkable formula of Jerison and Lee in CR geometry}
\author{Xiaodong Wang}
\address{Department of Mathematics, Michigan State University, East Lansing, MI 48824}
\email{xwang@math.msu.edu}

\begin{abstract}
We discuss a remarkable formula discovered by Jerison and Lee to classify
constant scalar curvature pseudohermitian structures on the sphere. We show
that the formula is valid in the wider context of Einstein pseudohermitian
manifolds. As an application we prove a uniqueness result that generalizes the
theorem of Jerison and Lee.

\end{abstract}
\maketitle

\section{\bigskip Introduction}

Recall that a Riemannian manifold $\left(  \Sigma^{n},g\right)  $ is called
Einstein if its Ricci curvature is constant, i.e. $Ric=cg$ for some constant
$c$. In this case the scalar curvature $R$ is then obviously a constant. In
general consider the trace-less Ricci tensor $T=Ric-\frac{R}{n}g$. By the 2nd
Bianchi identity we have%
\[
divT=\left(  \frac{1}{2}-\frac{1}{n}\right)  dR.
\]
As a corollary, we have the well-known fact that $\left(  \Sigma^{n},g\right)
$ with $n\geq3$ is Einstein iff $T=0$.

Given a closed Riemannian manifold $\left(  \Sigma^{n},g\right)  $ the famous
Yamabe problem seeks to conformally deform $g$ to get a new metric
$\widetilde{g}$ of constant scalar curvature. If we write $\widetilde
{g}=u^{4/\left(  n-2\right)  }g$, where $u$ is a positive smooth function,
then the scalar curvatures are related by the following equation%
\begin{equation}
-\frac{4\left(  n-1\right)  }{n-2}\Delta_{g}u+Ru=\widetilde{R}u^{\left(
n+2\right)  /\left(  n-2\right)  }. \label{confR}%
\end{equation}
The Yamabe problem was solved by Yamabe \cite{Y}, Trudinger \cite{T}, Aubin
\cite{A} and Schoen \cite{S} by showing that there is always a minimizer
$\overline{u}$ for the following variational problem%
\[
Y\left(  \Sigma,\left[  g\right]  \right)  :=\inf_{\substack{u\in C^{\infty
}\left(  \Sigma\right)  \\u>0}}\frac{\int_{\Sigma}\left(  \left\vert \nabla
u\right\vert ^{2}+Ru^{2}\right)  dv_{g}}{\left(  \int_{\Sigma}u^{2n/\left(
n-2\right)  }dv_{g}\right)  ^{\left(  n-2\right)  /n}}%
\]
as the metric $\overline{g}=\overline{u}^{4/\left(  n-2\right)  }g$ then has
constant scalar curvature. When $Y\left(  \Sigma,\left[  g\right]  \right)
\leq0$, we also have uniqueness: there is only one constant scalar curvature
metric up to scaling in the conformal class $\left[  g\right]  $. When
$Y\left(  \Sigma,\left[  g\right]  \right)  >0$, uniqueness in general fails.
But when there is an Einstein metric in the conformal class, we have the
following beautiful and important theorem due to Obata \cite{O2}.

\begin{theorem}
Suppose $\left(  \Sigma^{n},\widetilde{g}\right)  $ is a closed Einstein
manifold and $g=\phi\overline{g}$ is a conformal metric with constant scalar
curvature, where $\phi$ is a positive smooth function. Then

\begin{itemize}
\item $g$ is Einstein as well;

\item furthermore $\phi$ must be constant unless $\left(  \Sigma
^{n},\widetilde{g}\right)  $ is isometric to the standard sphere $\left(
\mathbb{S}^{n},g_{c}\right)  $ up to a scaling and $\phi$ corresponds to the
following function on $\mathbb{S}^{n}$%
\[
\phi\left(  x\right)  =c\left(  \cosh t+\sinh tx\cdot a\right)  ^{-2}%
\]
for some $c>0,t\geq0$ and $a\in\mathbb{S}^{n}$.
\end{itemize}
\end{theorem}

\bigskip When $n=2$ this is classic. Obata's proof for $n\geq3$ is very
elegant and is based on the following formula
\[
\widetilde{T}=T+\left(  n-2\right)  \phi^{-1}\left(  D^{2}\phi-\frac
{\Delta\phi}{n}g\right)  .
\]
Since $\widetilde{g}$ is Einstein, we have $\widetilde{T}=0$. Thus $-\left(
n-2\right)  \phi^{-1}\left(  D^{2}\phi-\frac{\Delta\phi}{n}g\right)  =T$.
Pairing with $T$ yields and using the fact the $g$ has constant scalar
curvature we obtain%
\[
-\left(  n-2\right)  div\left(  T\left(  \nabla\phi,\cdot\right)  \right)
=\phi\left\vert T\right\vert ^{2}.
\]

As a corollary we have the complete classification of positive solutions of a
nonlinear PDE (stated only for $n\geq3$ for brevity).

\begin{corollary}
On $\left(  \mathbb{S}^{n},g_{c}\right)  $ all positive solutions of the
equation
\[
-\frac{4}{n\left(  n-2\right)  }\Delta u+u=u^{\left(  n+2\right)  /\left(
n-2\right)  }%
\]
are of the form
\[
u\left(  x\right)  =\left(  \cosh t+\left(  \sinh t\right)  x\cdot a\right)
^{-\left(  n-2\right)  /2}%
\]
for some $t\geq0$ and $a\in\mathbb{S}^{n}$.
\end{corollary}

Equivalently one can through the stereographic projection consider the
following equation on $\mathbb{R}^{n}$
\[
-\Delta v=v^{\left(  n+2\right)  /\left(  n-2\right)  },v>0.
\]
This equation has been studied intensively from the PDE perspectively via the
method of moving planes or moving spheres, cf. \cite{GNN, CGS} and the more
recent \cite{LZ}.

In this paper we consider the analogues of these uniqueness results in CR
geometry. Let $\left(  M,\theta\right)  $ be a pseudohermitian manifold of
dimension $2m+1$ and $T$ the Reeb vector field. We always work with a local
unitary frame $\{T_{\alpha}:\alpha=1,\cdots,m\}$ for $T^{1,0}\left(  M\right)
$ and its dual frame $\left\{  \theta^{\alpha}\right\}  $. Thus%
\[
d\theta=\sqrt{-1}\sum_{\alpha}\theta^{\alpha}\wedge\overline{\theta^{\alpha}%
}.
\]
We will often denote $T$ by $T_{0}$. Let $B_{\alpha\overline{\beta}}%
=R_{\alpha\overline{\beta}}-\frac{R}{m}\delta_{\alpha\beta}$ be the trace-less
pseudohermitian Ricci tensor, where $R$ is the pseudohermitian scalar
curvature. We say that $\theta$ is pseudo-Einstein if $R_{\alpha
\overline{\beta}}=\frac{R}{m}\delta_{\alpha\beta}$ or $B_{\alpha
\overline{\beta}}=0$. This is always true when $m=1$. Pseudo-Einstein
manifolds were first introduced by Lee \cite{Lee}. By the Bianchi identity in
CR geometry we have%
\[
B_{\alpha\overline{\beta},\beta}=\left(  1-\frac{1}{m}\right)  R_{\alpha
}-\sqrt{-1}\left(  m-1\right)  A_{\alpha\beta,\overline{\beta}}%
\]
If $\theta$ is pseudo-Einstein and $m\geq2$ then%
\[
R_{\alpha}=\sqrt{-1}\left(  m-1\right)  A_{\alpha\beta,\overline{\beta}}.
\]
Therefore a pseudo-Einstein $\theta$ does not necessarily have constant scalar
curvature due to the presence of the torsion $A$. If the torsion $A$ vanishes,
then a pseudo-Einstein $\left(  M,\theta\right)  $ of dimension $2m+1\geq5$ is
of constant scalar curvature. Slightly more general, we have

\begin{proposition}
\label{bi}Suppose $\left(  M^{2m+1},\theta\right)  $ with $m\geq2$ is
pseudo-Einstein and its torsion has zero divergence. Then its scalar curvature
is constant.
\end{proposition}

\begin{proof}
By the formula above, $R_{\alpha}=0$. By taking conjugate, we also have
$R_{\overline{\beta}}=0$. Thus
\[
\sqrt{-1}\delta_{\alpha\overline{\beta}}R_{0}=R_{\alpha,\overline{\beta}%
}-R_{\overline{\beta},\alpha}=0.
\]
It follow $R_{0}=0$ as well and hence $R$ is constant.
\end{proof}

\begin{definition}
A pseudohermitian manifold $\left(  M^{2m+1},\theta\right)  $ is called
Einstein if it is torsion-free and the pseudohermitian Ricci tensor is
constant, i.e. $R_{\alpha\overline{\beta}}=\lambda\delta_{\alpha\beta}$ for
some constant $\lambda$.
\end{definition}

When $m\geq2$, $\left(  M^{2m+1},\theta\right)  $ is Einstein iff it is
pseudo-Einstein and torsion-free by Proposition \ref{bi}

If $\widetilde{\theta}=f^{2/m}\theta$ is another pseudohermitian structure,
where $f$ is a smooth and positive function, then the pseudohermitian scalar
curvatures of $\theta$ and $\widetilde{\theta}$ are related by the following
formula%
\[
-\frac{2\left(  m+1\right)  }{m}\Delta_{b}f+R=\widetilde{R}f^{\left(
m+2\right)  /m}.
\]

The CR Yamabe problem, initiated by Jerison and Lee \cite{JL1}, seeks to
conformally deform $\theta$ to get a new pseudohermitian structure
$\widetilde{\theta}$ of constant scalar curvature. Like the Riemannian case,
for a closed strictly pseudoconvex CR manifold $M^{2m+1}$ one considers the
Yamabe functional
\[
Y\left(  M,\theta\right)  =\frac{\int_{M}Rdv_{\theta}}{\left(  \int
_{M}dv_{\theta}\right)  ^{m/\left(  m+1\right)  }},
\]
where $\theta$ is any contact form associated to the CR structure and
$dv_{\theta}=\theta\wedge\left(  d\theta\right)  ^{m}$ is the volume form.
Set
\[
Y\left(  M\right)  =\inf_{\theta}Y\left(  M,\theta\right)  .
\]
This defines a CR invariant. The CR Yamabe problem, interpreted narrowly, is
whether the infimum is achieved. As in the Riemannian case, the unit sphere
$\mathbb{S}^{2m+1}=\left\{  z\in\mathbb{C}^{m+1}:\left\vert z\right\vert
=1\right\}  $ with its canonical pseudohermitian structure $\theta_{c}=\left(
2\sqrt{-1}\overline{\partial}\left\vert z\right\vert ^{2}\right)
|_{\mathbb{S}^{2m+1}}$ plays a fundamental role. $\left(  \mathbb{S}%
^{2m+1},\theta_{c}\right)  $ is of constant pseudohermitian curvature with
$R_{\alpha\overline{\beta}}=\left(  m+1\right)  /2\delta_{\alpha\beta}$.

In their fundamental work \cite{JL1}, Jerison and Lee proved the following

\begin{enumerate}
\item $Y\left(  \mathbb{S}^{2m+1}\right)  =Y\left(  \mathbb{S}^{2m+1}%
,\theta_{c}\right)  =2\pi m\left(  m+1\right)  $, or equivalently the
following sharp Sobolev inequality holds on $\mathbb{S}^{2m+1}$%
\begin{align}
&  \int_{\mathbb{S}^{2m+1}}\left(  \frac{2\left(  m+1\right)  }{m}\left\vert
\nabla_{b}f\right\vert ^{2}+\frac{m\left(  m+1\right)  }{2}f^{2}\right)
dv_{c}\label{Sob}\\
&  \geq2\pi m\left(  m+1\right)  \left(  \int_{\mathbb{S}^{2m+1}}\left\vert
f\right\vert ^{2\left(  m+1\right)  /m}dv_{c}\right)  ^{m/\left(  m+1\right)
},\nonumber
\end{align}
where $\nabla_{b}f$ is the horizontal gradient and $dv_{c}=\theta_{c}%
\wedge\left(  d\theta_{c}\right)  ^{m}$

\item For any closed strictly pseudoconvex CR manifold $M^{2m+1}$%
\[
Y\left(  M\right)  \leq Y\left(  \mathbb{S}^{2m+1}\right)  ;
\]

\item The CR Yamabe problem has a solution if $Y\left(  M\right)  <Y\left(
\mathbb{S}^{2m+1}\right)  $.
\end{enumerate}

Moreover, they proved $Y\left(  M\right)  <Y\left(  \mathbb{S}^{2m+1}\right)
$ when $m\geq2$ and $M$ is not locally CR equivalent to $\mathbb{S}^{2m+1}$.
To our best knowledge, the conjecture that $Y\left(  M\right)  <Y\left(
\mathbb{S}^{2m+1}\right)  $ unless $M$ is CR equivalent to $\mathbb{S}^{2m+1}$
is still open in the remaining cases. However, it is now known that for all
compact strictly pseudoconvex CR manifold $M$ there is always a
pseudohermitian structure $\theta$ on $M$ whose scalar curvature is constant
by the more recent work Gamara \cite{G} and Gamara-Yacoub \cite{GY}.

Similar to the Riemannian case, there is a unique constant scalar curvature
pseudohermitian structure on $M$ up to scaling when $Y\left(  M\right)  \leq
0$. But uniqueness in general fails if $Y\left(  M\right)  >0$. For the CR
sphere $\mathbb{S}^{2m+1}$ Jerison and Lee \cite{JL2} classified all
pseudohermitian structures with constant scalar curvature. This is of
fundamental importance for the whole program of CR Yamabe problem.

\begin{theorem}
Suppose $\theta=\phi\theta_{c}$ is a pseudohermitian structure on
$\mathbb{S}^{2m+1}$. Then $\theta$ has constant scalar curvature iff $\phi$ is
of the following form:%
\[
\phi\left(  z\right)  =c\left\vert \cosh t+\left(  \sinh t\right)
z\cdot\overline{\xi}\right\vert ^{-2}%
\]
for some $c>0,t\geq0$ and $\xi\in\mathbb{S}^{2m+1}$.
\end{theorem}

\bigskip The key ingredient in the proof is the following remarkable, highly
nontrivial identity on $\left(  \mathbb{S}^{2m+1},\theta\right)  $:
\begin{align*}
&  \operatorname{Re}\left(  gD_{\alpha}+\overline{g}E_{\alpha}-3\phi_{0}%
\sqrt{-1}U_{\alpha}\right)  _{,\overline{\alpha}}\\
&  =\left(  \frac{1}{2}+\frac{1}{2}\phi\right)  \left(  \left\vert
D_{\alpha\beta}\right\vert ^{2}+\left\vert E_{\alpha\overline{\beta}%
}\right\vert ^{2}\right) \\
&  +\phi\left[  \left\vert D_{\alpha}-U_{\alpha}\right\vert ^{2}+\left\vert
U_{\alpha}+E_{\alpha}-D_{\alpha}\right\vert ^{2}+\left\vert U_{\alpha
}+E_{\alpha}\right\vert ^{2}+\left\vert \phi^{-1}\phi_{\overline{\gamma}%
}D_{\alpha\beta}+\phi^{-1}\phi_{\beta}E_{\alpha\overline{\gamma}}\right\vert
^{2}\right]  .
\end{align*}
where%
\begin{align*}
D_{\alpha\beta}  &  =\phi^{-1}\phi_{\alpha,\beta},D_{\alpha}=\phi^{-1}%
\phi_{\overline{\beta}}D_{\alpha\beta},E_{\alpha}=\phi^{-1}\phi_{\gamma
}E_{\alpha\overline{\gamma}},\\
E_{\alpha\overline{\beta}}  &  =\phi^{-1}\phi_{\alpha,\overline{\beta}}%
-\phi^{-2}\phi_{\alpha}\phi_{\overline{\beta}}-\frac{1}{2}\left(  \frac{1}%
{2}\phi^{-1}-\frac{1}{2}+\phi^{-2}\left\vert \partial\phi\right\vert
^{2}+\sqrt{-1}\phi^{-1}\phi_{0}\right)  \delta_{\alpha\overline{\beta}},\\
U_{\alpha}  &  =-\frac{2}{m+2}\sqrt{-1}A_{\alpha\beta,\overline{\beta}%
},g=\frac{1}{2}+\frac{1}{2}\phi+\phi^{-1}\left\vert \partial\phi\right\vert
^{2}+\sqrt{-1}\phi_{0}.
\end{align*}

We should also mention the recent deep work \cite{FL} by Frank and Lieb in
which the sharp Hardy-Littlewood-Sobolev inequality, of which (\ref{Sob}) is a
special case, on the Heisenberg group is established. Their paper also
contains a new and shorter proof of the Sobolev inequality (\ref{Sob}) as well
as a nice argument which yields the classification of all the minimizers quickly.

The purpose of this work is to point out that the Jerison-Lee identity is
valid on any closed Einstein pseudohermitian manifold $\left(  M^{2m+1}%
,\theta\right)  $ and as an application prove the following uniqueness theorem
which generalizes the above result of Jerison and Lee.

\begin{theorem}
Let $\left(  M^{2m+1},\widetilde{\theta}\right)  $ be a closed Einstein
pseudohermitian manifold. Suppose $\theta=\phi\widetilde{\theta}$ is another
pseudohermitian structure with constant pseudohermitian scalar curvature. Then

\begin{itemize}
\item $\theta$ is Einstein as well;

\item furthermore $\phi$ must be constant unless $\left(  M^{2m+1}%
,\widetilde{\theta}\right)  $ is CR isometric to $\left(  \mathbb{S}%
^{2m+1},\theta_{c}\right)  $ up to a scaling and $\phi$ corresponds to the
following function on $\mathbb{S}^{2m+1}$%
\[
\phi\left(  z\right)  =c\left\vert \cosh t+\left(  \sinh t\right)
z\cdot\overline{\xi}\right\vert ^{-2}%
\]
for some $c>0,t\geq0$ and $\xi\in\mathbb{S}^{2m+1}$.
\end{itemize}
\end{theorem}

\bigskip More precisely, the second part means that if $\phi$ is not constant
then there exists a CR diffeomorphism $F:M\rightarrow\mathbb{S}^{2m+1}$ s.t.
$F^{\ast}\theta_{c}=\lambda\widetilde{\theta}$ for some $\lambda>0$ and
$\phi\circ F^{-1}$ is of the form above on $\mathbb{S}^{2m+1}$.

The paper is organized as follows. In Section 2 we present the Jerison-Lee
identity on any closed Einstein pseudohermitian manifold $\left(
M^{2m+1},\theta\right)  $. Using this identity we will prove the above theorem
in Section 3. There is an appendix in which we collect several formulas in CR
geometry that are needed in Section 3.

\textbf{Acknowledgements. }The author is very grateful to Professor Jerison
and Professor Lee for kindly helping him understand their work \cite{JL2}. He
also wishes to thanks Song-Ying Li and Meijun Zhu for useful discussions.

\section{The Jerison-Lee identity}

In this Section we discuss the Jerison-Lee identity from \cite{JL2}. Though it
is only stated for the CR sphere $\mathbb{S}^{2m+1}$ there, the identity and
its proof are valid on any closed Einstein pseudohermitian manifold. To
publicize this remarkable identity in the wider context and also for
completeness, we present a detailed proof following \cite{JL2} faithfully.
Therefore this Section is expository. We use slightly different notation and
provide more details at several places to make the proof easier to follow.

Let $\left(  M^{2m+1},\theta\right)  $ be a pseudohermitian manifold and
$\phi$ a smooth and positive function. Consider $\widetilde{\theta}=\phi
^{-1}\theta$. The pseudohermitian invariants transform as follows
(\cite{Lee}):
\begin{align*}
\widetilde{A}_{\alpha\beta}  &  = A_{\alpha\beta}-\sqrt{-1}\phi^{-1}%
\phi_{\alpha,\beta}\\
\widetilde{B}_{\alpha\overline{\beta}}  &  = B_{\alpha\overline{\beta}%
}+\left(  m+2\right)  \left(  \phi^{-1}\phi_{\alpha,\overline{\beta}}%
-\phi^{-2}\phi_{\alpha}\phi_{\overline{\beta}}\right)  -\frac{m+2}{m}\left(
\phi^{-1}\phi_{\gamma,\overline{\gamma}}-\phi^{-2}\left\vert \partial
\phi\right\vert ^{2}\right)  h_{\alpha\overline{\beta}}\\
\widetilde{R}  &  = \phi R+\left(  m+1\right)  \Delta_{b}\phi-m\left(
m+1\right)  \phi^{-1}\left\vert \partial\phi\right\vert ^{2}.
\end{align*}
here we are working with a local unitary frame $\{T_{\alpha}:\alpha
=1,\cdots,m\}$ w.r.t. $\theta$.

\begin{proposition}
\label{JLI}Let $\theta$ and $\widetilde{\theta}=\phi^{-1}\theta$ be two
pseudohermitian structures on a closed manifold $M^{2m+1}$. Suppose that both
$\theta$ and $\widetilde{\theta}$ have constant scalar curvature $m\left(
m+1\right)  /2$ and $\widetilde{\theta}$ is Einstein. Set%
\begin{align*}
D_{\alpha\beta}  &  =-\sqrt{-1}A_{\alpha\beta},D_{\alpha}=\phi^{-1}%
\phi_{\overline{\beta}}D_{\alpha\beta},\\
E_{\alpha\overline{\beta}}  &  =-\frac{1}{m+2}B_{\alpha\overline{\beta}%
},E_{\alpha}=\phi^{-1}\phi_{\beta}E_{\alpha\overline{\beta}},\\
U_{\alpha}  &  =-\frac{2}{m+2}\sqrt{-1}A_{\alpha\beta,\overline{\beta}},\\
g  &  =\frac{1}{2}+\frac{1}{2}\phi+\phi^{-1}\left\vert \partial\phi\right\vert
^{2}+\sqrt{-1}\phi_{0}.
\end{align*}
Then%
\begin{align}
&  \operatorname{Re}\left(  gD_{\alpha}+\overline{g}E_{\alpha}-3\phi_{0}%
\sqrt{-1}U_{\alpha}\right)  _{,\overline{\alpha}}\label{JLE}\\
&  =\left(  \frac{1}{2}+\frac{1}{2}\phi\right)  \left(  \left\vert
D_{\alpha\beta}\right\vert ^{2}+\left\vert E_{\alpha\overline{\beta}%
}\right\vert ^{2}\right) \nonumber\\
&  +\phi\left[  \left\vert D_{\alpha}-U_{\alpha}\right\vert ^{2}+\left\vert
U_{\alpha}+E_{\alpha}-D_{\alpha}\right\vert ^{2}+\left\vert U_{\alpha
}+E_{\alpha}\right\vert ^{2}+\left\vert \phi^{-1}\phi_{\overline{\gamma}%
}D_{\alpha\beta}+\phi^{-1}\phi_{\beta}E_{\alpha\overline{\gamma}}\right\vert
^{2}\right]  .\nonumber
\end{align}

\end{proposition}

\begin{remark}
\bigskip Jerison-Lee \cite{JL2} use the normalization $R=\widetilde
{R}=m\left(  m+1\right)  $. We instead use the normalization $R=\widetilde
{R}=m\left(  m+1\right)  /2$. This is why some of our coefficients are
different. Our normalization has the advantage that the adapted metric for
$\theta_{c}$ on $\mathbb{S}^{2m+1}$ is round.
\end{remark}

Since $\widetilde{\theta}$ is Einstein and $R=\widetilde{R}=m\left(
m+1\right)  /2$, we have
\begin{align}
A_{\alpha\beta}  &  =\sqrt{-1}\phi^{-1}\phi_{\alpha,\beta},\label{FA}\\
B_{\alpha\overline{\beta}}  &  =-\left(  m+2\right)  \left(  \phi^{-1}%
\phi_{\alpha,\overline{\beta}}-\phi^{-2}\phi_{\alpha}\phi_{\overline{\beta}%
}\right)  +\frac{m+2}{m}\left(  \phi^{-1}\phi_{\gamma,\overline{\gamma}}%
-\phi^{-2}\left\vert \partial\phi\right\vert ^{2}\right)  h_{\alpha
\overline{\beta}},\label{FB}\\
\phi_{\gamma,\overline{\gamma}}  &  =\frac{m}{4}-\frac{m}{4}\phi+\frac{m+2}%
{2}\phi^{-1}\left\vert \partial\phi\right\vert ^{2}+\frac{m}{2}\sqrt{-1}%
\phi_{0}. \label{FR}%
\end{align}
Thus%
\begin{align}
D_{\alpha\beta}  &  =\phi^{-1}\phi_{\alpha,\beta},\label{FD}\\
E_{\alpha\overline{\beta}}  &  =\phi^{-1}\phi_{\alpha,\overline{\beta}}%
-\phi^{-2}\phi_{\alpha}\phi_{\overline{\beta}}-\frac{1}{2}\left(  \frac{1}%
{2}\phi^{-1}-\frac{1}{2}+\phi^{-2}\left\vert \partial\phi\right\vert
^{2}+\sqrt{-1}\phi^{-1}\phi_{0}\right)  \delta_{\alpha\overline{\beta}}.
\label{FE}%
\end{align}
We compute%
\begin{align*}
\phi_{\alpha\beta,\overline{\beta}}  &  =\phi_{\beta,\overline{\beta}\alpha
}+\sqrt{-1}\delta_{\alpha\overline{\beta}}\phi_{\beta,0}+R_{\alpha
\overline{\beta}}\phi_{\beta}\\
&  =-\frac{m}{4}\phi_{\alpha}-\frac{m+2}{2}\phi^{-2}\left\vert \partial
\phi\right\vert ^{2}+\frac{m+2}{2}\phi^{-1}\phi_{\overline{\beta}}\phi
_{\alpha,\beta}+\frac{m+2}{2}\phi^{-1}\phi_{\overline{\beta},\alpha}%
\phi_{\beta}\\
&  +\frac{m}{2}\sqrt{-1}\phi_{0,\alpha}+\sqrt{-1}\phi_{\alpha,0}%
+R_{\alpha\overline{\beta}}\phi_{\beta}\\
&  =-\frac{m}{4}\phi_{\alpha}-\frac{m+2}{2}\phi^{-2}\left\vert \partial
\phi\right\vert ^{2}\phi_{\alpha}+\frac{m+2}{2}\phi^{-1}\phi_{\overline{\beta
}}\phi_{\alpha,\beta}+\frac{m+2}{2}\phi^{-1}\phi_{\alpha,\overline{\beta}}%
\phi_{\beta}\\
&  -\sqrt{-1}\frac{m+2}{2}\phi^{-1}\phi_{\alpha}\phi_{0}+\frac{m+2}{2}%
\sqrt{-1}\phi_{0,\alpha}-\sqrt{-1}A_{\alpha\beta}\phi_{\overline{\beta}%
}+R_{\alpha\overline{\beta}}\phi_{\beta}%
\end{align*}
Using the decomposition $R_{\alpha\overline{\beta}}=B_{\alpha\overline{\beta}%
}-\frac{R}{m}\delta_{\alpha\beta}$, (\ref{FB}) and the fact $R=m\left(
m+1\right)  /2$ and simplifying we obtain%
\[
R_{\alpha\overline{\beta}}\phi_{\beta}=\frac{m+2}{2}\left(  -2\phi^{-1}%
\phi_{\alpha,\overline{\beta}}\phi_{\beta}+3\phi^{-2}\left\vert \partial
\phi\right\vert ^{2}\phi_{\alpha}+\frac{1}{2}\phi^{-1}\phi_{\alpha}+\sqrt
{-1}\phi^{-1}\phi_{0}\phi_{\alpha}\right)  +\frac{m}{4}\phi_{\alpha}.
\]
Plugging it into the previous formula yields%
\begin{align*}
\phi_{\alpha\beta,\overline{\beta}}  &  =-\frac{m}{4}\phi_{\alpha}-\frac
{m+2}{2}\phi^{-2}\left\vert \partial\phi\right\vert ^{2}\phi_{\alpha}%
+\frac{m+2}{2}\phi^{-1}\phi_{\overline{\beta}}\phi_{\alpha,\beta}+\frac
{m+2}{2}\phi^{-1}\phi_{\alpha,\overline{\beta}}\phi_{\beta}\\
&  -\sqrt{-1}\frac{m+2}{2}\phi^{-1}\phi_{\alpha}\phi_{0}+\frac{m+2}{2}%
\sqrt{-1}\phi_{0,\alpha}-\sqrt{-1}A_{\alpha\beta}\phi_{\overline{\beta}}\\
&  +\frac{m+2}{2}\left(  -2\phi^{-1}\phi_{\alpha,\overline{\beta}}\phi_{\beta
}+3\phi^{-2}\left\vert \partial\phi\right\vert ^{2}\phi_{\alpha}+\frac{1}%
{2}\phi^{-1}\phi_{\alpha}+\sqrt{-1}\phi^{-1}\phi_{0}\phi_{\alpha}\right)
+\frac{m}{4}\phi_{\alpha}\\
&  =\phi^{-1}\phi_{\alpha,\beta}\phi_{\overline{\beta}}+\frac{m+2}{2}\left(
\sqrt{-1}\phi_{0,\alpha}+\phi^{-1}\phi_{\alpha,\beta}\phi_{\overline{\beta}%
}-\phi^{-1}\phi_{\alpha,\overline{\beta}}\phi_{\beta}+2\phi^{-2}\left\vert
\partial\phi\right\vert ^{2}\phi_{\alpha}+\frac{1}{2}\phi^{-1}\phi_{\alpha
}\right)
\end{align*}
Therefore%
\begin{align}
U_{\alpha}  &  =-\frac{2}{m+2}\sqrt{-1}A_{\alpha\beta,\overline{\beta}%
}\label{Ua1}\\
&  =\phi^{-1}\left[  \sqrt{-1}\phi_{0,\alpha}+2\phi^{-2}\left\vert
\partial\phi\right\vert ^{2}\phi_{\alpha}+\phi^{-1}\left(  \phi_{\alpha,\beta
}\phi_{\overline{\beta}}-\phi_{\alpha,\overline{\beta}}\phi_{\beta}+\frac
{1}{2}\phi_{\alpha}\right)  \right] \nonumber\\
&  =\phi^{-1}\left[  \sqrt{-1}\phi_{0,\alpha}+2\phi^{-2}\left\vert
\partial\phi\right\vert ^{2}\phi_{\alpha}+\phi D_{\alpha}-\phi^{-1}%
\phi_{\alpha,\overline{\beta}}\phi_{\beta}+\frac{1}{2}\phi^{-1}\phi_{\alpha
}\right] \nonumber
\end{align}

\begin{lemma}
\bigskip We have%
\begin{equation}
U_{\alpha}=\phi^{-1}\left[  \sqrt{-1}\phi_{0,\alpha}+\phi\left(  D_{\alpha
}-E_{\alpha}\right)  +\frac{1}{2}\phi^{-1}\overline{g}\phi_{\alpha}\right]  .
\label{Ua}%
\end{equation}

\end{lemma}

\begin{proof}
Replacing $\phi_{\alpha,\overline{\beta}}$ in (\ref{Ua1}) using identity
(\ref{Eab}) we obtain%
\begin{align*}
U_{\alpha}  &  =\phi^{-1}\left[  \sqrt{-1}\phi_{0,\alpha}+2\phi^{-2}\left\vert
\partial\phi\right\vert ^{2}\phi_{\alpha}+\phi D_{\alpha}+\frac{1}{2}\phi
^{-1}\phi_{\alpha}\right. \\
&  \left.  -E_{\alpha,\overline{\beta}}\phi_{\beta}-\phi^{-2}\left\vert
\partial\phi\right\vert ^{2}\phi_{\alpha}-\frac{1}{2}\left(  \frac{1}{2}%
\phi^{-1}-\frac{1}{2}+\phi^{-2}\left\vert \partial\phi\right\vert ^{2}%
+\sqrt{-1}\phi^{-1}\phi_{0}\right)  \phi_{\alpha}\right] \\
&  =\phi^{-1}\left[  \sqrt{-1}\phi_{0,\alpha}+\phi\left(  D_{\alpha}%
-E_{\alpha}\right)  +\frac{1}{2}\left(  \frac{1}{2}+\frac{1}{2}\phi^{-1}%
+\phi^{-2}\left\vert \partial\phi\right\vert ^{2}-\sqrt{-1}\phi^{-1}\phi
_{0}\right)  \phi_{\alpha}\right] \\
&  =\phi^{-1}\left[  \sqrt{-1}\phi_{0,\alpha}+\phi\left(  D_{\alpha}%
-E_{\alpha}\right)  +\frac{1}{2}\phi^{-1}\overline{g}\phi_{\alpha}\right]  .
\end{align*}

\end{proof}

\bigskip

As $D_{\alpha}=-\sqrt{-1}A_{\alpha\gamma}\phi^{-1}\phi_{\overline{\gamma}}$,
we have%
\begin{align}
D_{\alpha,\overline{\alpha}}  &  =-\sqrt{-1}A_{\alpha\gamma,\overline{\alpha}%
}\phi^{-1}\phi_{\overline{\gamma}}-\sqrt{-1}A_{\alpha\gamma}\phi^{-1}%
\phi_{\overline{\alpha}\overline{\gamma}}+\sqrt{-1}A_{\alpha\gamma}\phi
^{-2}\phi_{\overline{\alpha}}\phi_{\overline{\gamma}}\label{DD}\\
&  =\frac{m+2}{2}\phi^{-1}U_{\alpha}\phi_{\overline{\alpha}}+\left\vert
D_{\alpha\beta}\right\vert ^{2}-D_{\alpha\beta}\phi^{-2}\phi_{\overline
{\alpha}}\phi_{\overline{\beta}}.\nonumber\\
&  =\frac{m+2}{2}\phi^{-1}U_{\alpha}\phi_{\overline{\alpha}}+\left\vert
D_{\alpha\beta}\right\vert ^{2}-\phi^{-1}D_{\alpha}\phi_{\overline{\alpha}%
}.\nonumber
\end{align}
As $E_{\alpha}=\phi^{-1}\phi_{\gamma}E_{\alpha\overline{\gamma}}$, we have%
\begin{align}
E_{\alpha,\overline{\alpha}}  &  =\phi^{-1}\phi_{\gamma}E_{\alpha
\overline{\gamma},\overline{\alpha}}+\left(  \phi^{-1}\phi_{\gamma
,\overline{\alpha}}-\phi^{-2}\phi_{\overline{\alpha}}\phi_{\gamma}\right)
E_{\alpha\overline{\gamma}}\label{DE}\\
&  =\frac{1-m}{2}\phi^{-1}\phi_{\gamma}U_{\overline{\gamma}}+\left(  \phi
^{-1}\phi_{\gamma,\overline{\alpha}}-\phi^{-2}\phi_{\overline{\alpha}}%
\phi_{\gamma}\right)  E_{\alpha\overline{\gamma}}\nonumber\\
&  =\frac{1-m}{2}\phi^{-1}\phi_{\gamma}U_{\overline{\gamma}}+\left\vert
E_{\alpha\overline{\beta}}\right\vert ^{2}\nonumber
\end{align}
We now compute the left hand side of (\ref{JLE})
\begin{align}
LHS  &  =\operatorname{Re}\left[  gD_{\alpha,\overline{\alpha}}+\overline
{g}E_{\alpha,\overline{\alpha}}\right]  -3\phi_{0}\operatorname{Re}\sqrt
{-1}U_{\alpha,\overline{\alpha}}\label{LHS1}\\
&  +\operatorname{Re}\left[  g_{\overline{\alpha}}D_{\alpha}+\overline
{g}_{\overline{\alpha}}E_{\alpha}-3\sqrt{-1}\phi_{0,\overline{\alpha}%
}U_{\alpha}\right]  .\nonumber
\end{align}
By a Bianchi identity \cite[(2.13)]{Lee} we have $2\operatorname{Re}%
A_{\alpha\beta,\overline{\beta}\overline{\alpha}}=R_{0}=0$. Thus
\[
\operatorname{Re}\sqrt{-1}U_{\alpha,\overline{\alpha}}=\frac{2}{m+2}%
\operatorname{Re}A_{\alpha\beta,\overline{\beta}\overline{\alpha}}=0.
\]
Plugging this identity as well as (\ref{DD}) and (\ref{DE}) into (\ref{LHS1})
we obtain
\begin{align}
LHS  &  =\frac{3}{2}\phi^{-1}\operatorname{Re}gU_{\alpha}\phi_{\overline
{\alpha}}\label{LHS2}\\
&  +\left(  \frac{1}{2}+\frac{1}{2}\phi+\phi^{-1}\left\vert \partial
\phi\right\vert ^{2}\right)  \left(  \left\vert D_{\alpha\beta}\right\vert
^{2}+\left\vert E_{\alpha\overline{\beta}}\right\vert ^{2}\right)  -\phi
^{-1}\operatorname{Re}gD_{\alpha}\phi_{\overline{\alpha}}\nonumber\\
&  +\operatorname{Re}\left[  g_{\overline{\alpha}}D_{\alpha}+\overline
{g}_{\overline{\alpha}}E_{\alpha}-3\sqrt{-1}\phi_{0,\overline{\alpha}%
}U_{\alpha}\right]  .\nonumber
\end{align}

\bigskip

\begin{lemma}
\bigskip We have
\begin{align*}
\sqrt{-1}\phi_{0,\overline{\alpha}}  &  =\phi\left(  D_{\overline{\alpha}%
}-U_{\overline{\alpha}}-E_{\overline{\alpha}}\right)  +\frac{1}{2}\phi
^{-1}g\phi_{\overline{\alpha}},\\
g_{\overline{\alpha}}  &  =\phi^{-1}g\phi_{\overline{\alpha}}+\phi\left(
2D_{\overline{\alpha}}-U_{\overline{\alpha}}\right)  ,\\
\overline{g}_{\overline{\alpha}}  &  =\phi\left(  2E_{\overline{\alpha}%
}+U_{\overline{\alpha}}\right)  .
\end{align*}

\end{lemma}

\begin{proof}
The first formula follows from (\ref{FE}) directly. We have%
\begin{align*}
g_{\overline{\alpha}}  &  =\frac{1}{2}\phi_{\overline{\alpha}}-\phi^{-2}%
\phi_{\overline{\alpha}}\left\vert \partial\phi\right\vert ^{2}+\phi
^{-1}\left(  \phi_{\beta\overline{\alpha}}\phi_{\overline{\beta}}+\phi_{\beta
}\phi_{\overline{\beta}\overline{\alpha}}\right)  +\sqrt{-1}\phi
_{0,\overline{\alpha}}\\
&  =\left(  \frac{1}{2}-\phi^{-2}\left\vert \partial\phi\right\vert
^{2}\right)  \phi_{\overline{\alpha}}+\phi D_{\overline{\alpha}}+\phi^{-1}%
\phi_{\beta\overline{\alpha}}\phi_{\overline{\beta}}+\sqrt{-1}\phi
_{0,\overline{\alpha}}.
\end{align*}
Since
\begin{align*}
\phi_{\beta\overline{\alpha}}\phi_{\overline{\beta}}  &  =\overline
{\phi_{\overline{\beta}\alpha}\phi_{\beta}}\\
&  =\overline{\left(  \phi_{\alpha\overline{\beta}}-\sqrt{-1}\phi_{0}%
\delta_{\alpha\beta}\right)  \phi_{\beta}}\\
&  =\overline{\phi_{\alpha\overline{\beta}}\phi_{\beta}}+\sqrt{-1}\phi_{0}%
\phi_{\overline{\alpha}}.
\end{align*}
Thus%
\[
g_{\overline{\alpha}}=\left(  \frac{1}{2}-\phi^{-2}\left\vert \partial
\phi\right\vert ^{2}+\sqrt{-1}\phi^{-1}\phi_{0}\right)  \phi_{\overline
{\alpha}}+\phi D_{\overline{\alpha}}+\phi^{-1}\overline{\phi_{\alpha
\overline{\beta}}\phi_{\beta}}+\sqrt{-1}\phi_{0,\overline{\alpha}}.
\]
Replacing $\phi^{-1}\phi_{\alpha,\overline{\beta}}$ by the formula we end up
with%
\begin{align*}
g_{\overline{\alpha}}  &  =\left(  \frac{1}{2}-\phi^{-2}\left\vert
\partial\phi\right\vert ^{2}+\sqrt{-1}\phi^{-1}\phi_{0}\right)  \phi
_{\overline{\alpha}}+\phi\left(  D_{\overline{\alpha}}+E_{\overline{\alpha}%
}\right)  +\sqrt{-1}\phi_{0,\overline{\alpha}}\\
&  +\overline{\phi^{-2}\left\vert \partial\phi\right\vert ^{2}\phi_{\alpha
}+\frac{1}{2}\left(  \frac{1}{2}\phi^{-1}-\frac{1}{2}+\phi^{-2}\left\vert
\partial\phi\right\vert ^{2}+\sqrt{-1}\phi^{-1}\phi_{0}\right)  \phi_{\alpha}%
}\\
&  =\frac{1}{2}\left(  \frac{1}{2}+\frac{1}{2}\phi^{-1}+\phi^{-2}\left\vert
\partial\phi\right\vert ^{2}+\sqrt{-1}\phi^{-1}\phi_{0}\right)  \phi
_{\overline{\alpha}}+\phi\left(  D_{\overline{\alpha}}+E_{\overline{\alpha}%
}\right)  +\sqrt{-1}\phi_{0,\overline{\alpha}}\\
&  =\frac{1}{2}\phi^{-1}g\phi_{\overline{\alpha}}+\phi\left(  D_{\overline
{\alpha}}+E_{\overline{\alpha}}\right)  +\sqrt{-1}\phi_{0,\overline{\alpha}}%
\end{align*}
By the same calculation%
\[
\overline{g}_{\overline{\alpha}}=\frac{1}{2}\phi^{-1}g\phi_{\overline{\alpha}%
}+\phi\left(  D_{\overline{\alpha}}+E_{\overline{\alpha}}\right)  -\sqrt
{-1}\phi_{0,\overline{\alpha}}.
\]
Plugging the first formula into the above identities, we obtain the second and
third formulas.
\end{proof}

\bigskip

\bigskip Plugging these formulas into (\ref{LHS2}) we obtain
\begin{align*}
LHS  &  =\frac{3}{2}\phi^{-1}\operatorname{Re}gU_{\alpha}\phi_{\overline
{\alpha}}+\left(  \frac{1}{2}+\frac{1}{2}\phi+\phi^{-1}\left\vert \partial
\phi\right\vert ^{2}\right)  \left(  \left\vert D_{\alpha\beta}\right\vert
^{2}+\left\vert E_{\alpha\overline{\beta}}\right\vert ^{2}\right) \\
&  -\phi^{-1}\operatorname{Re}gD_{\alpha}\phi_{\overline{\alpha}%
}+\operatorname{Re}\left[  \phi^{-1}g\phi_{\overline{\alpha}}+\phi\left(
2D_{\overline{\alpha}}-U_{\overline{\alpha}}\right)  \right]  D_{\alpha}\\
&  +\operatorname{Re}\left[  \phi\left(  2E_{\overline{\alpha}}+U_{\overline
{\alpha}}\right)  \right]  E_{\alpha}-3\operatorname{Re}\left[  \phi\left(
D_{\overline{\alpha}}-U_{\overline{\alpha}}-E_{\overline{\alpha}}\right)
+\frac{1}{2}\phi^{-1}g\phi_{\overline{\alpha}}\right]  U_{\alpha}\\
&  =\left(  \frac{1}{2}+\frac{1}{2}\phi+\phi^{-1}\left\vert \partial
\phi\right\vert ^{2}\right)  \left(  \left\vert D_{\alpha\beta}\right\vert
^{2}+\left\vert E_{\alpha\overline{\beta}}\right\vert ^{2}\right) \\
&  +\phi\operatorname{Re}\left(  2D_{\overline{\alpha}}-U_{\overline{\alpha}%
}\right)  D_{\alpha}+\phi\operatorname{Re}\left(  2E_{\overline{\alpha}%
}+U_{\overline{\alpha}}\right)  E_{\alpha}+3\phi\operatorname{Re}\left(
U_{\overline{\alpha}}+E_{\overline{\alpha}}-D_{\overline{\alpha}}\right)
U_{\alpha}.
\end{align*}
It is then elementary to show that this equals the RHS.

\section{Proof of the main theorem}

We are now ready to prove our main theorem.

\begin{theorem}
Let $\left(  M^{2m+1},\widetilde{\theta}\right)  $ be a closed Einstein
pseudohermitian manifold. Suppose $\theta=\phi\widetilde{\theta}$ is another
pseudohermitian structure with constant pseudohermitian scalar curvature. Then

\begin{itemize}
\item $\theta$ is Einstein as well;

\item furthermore $\phi$ must be constant unless $\left(  M^{2m+1}%
,\widetilde{\theta}\right)  $ is CR isometric to the standard sphere $\left(
\mathbb{S}^{2m+1},\theta_{c}\right)  $ up to a scaling and $\phi$ corresponds
to the following function on $\mathbb{S}^{2m+1}$%
\[
\phi\left(  z\right)  =c\left\vert \cosh t+\left(  \sinh t\right)
z\cdot\overline{\xi}\right\vert ^{-2}%
\]
for some $c>0,t\geq0$ and $\xi\in\mathbb{S}^{2m+1}$.
\end{itemize}
\end{theorem}

The Theorem is trivial if the pseudohermitian scalar curvature\ of
$\widetilde{\theta}$ is zero or negative. Assume it is positive. By scaling
both $\theta$ and $\widetilde{\theta}$, we may assume $R=\widetilde
{R}=m\left(  m+1\right)  /2$. Integrating the Jerison-Lee identity (\ref{JLE})
over $M$ we have%
\begin{align*}
0  &  =\int_{M}\left(  \frac{1}{2}+\frac{1}{2}\phi\right)  \left(  \left\vert
D_{\alpha\beta}\right\vert ^{2}+\left\vert E_{\alpha\overline{\beta}%
}\right\vert ^{2}\right)  dv_{\theta}\\
&  +\int_{M}\phi\left[  \left\vert D_{\alpha}-U_{\alpha}\right\vert
^{2}+\left\vert U_{\alpha}+E_{\alpha}-D_{\alpha}\right\vert ^{2}+\left\vert
U_{\alpha}+E_{\alpha}\right\vert ^{2}+\left\vert \phi^{-1}\phi_{\overline
{\gamma}}D_{\alpha\beta}+\phi^{-1}\phi_{\beta}E_{\alpha\overline{\gamma}%
}\right\vert ^{2}\right]  dv_{\theta}.
\end{align*}
Therefore%
\[
D_{\alpha\beta}=0,E_{\alpha\overline{\beta}}=0,U_{\alpha}=0,
\]
i.e. more explicitly
\begin{align*}
\phi_{\alpha,\beta}  &  =0,\\
\phi_{\alpha,\overline{\beta}}  &  =\phi^{-1}\phi_{\alpha}\phi_{\overline
{\beta}}+\frac{1}{2}\left(  \frac{1}{2}-\frac{1}{2}\phi+\phi^{-1}\left\vert
\partial\phi\right\vert ^{2}+\sqrt{-1}\phi_{0}\right)  \delta_{\alpha
\overline{\beta}},\\
\phi_{0,\alpha}  &  =\frac{\sqrt{-1}}{2}\left(  \frac{1}{2}+\frac{1}{2}%
\phi^{-1}+\phi^{-2}\left\vert \partial\phi\right\vert ^{2}-\sqrt{-1}\phi
^{-1}\phi_{0}\right)  \phi_{\alpha}.
\end{align*}
Recall that $D_{\alpha\beta}=-\sqrt{-1}A_{\alpha\beta}$ and $E_{\alpha
\overline{\beta}}$ is a multiple of the trace-less Ricci by definition.
Therefore $\theta$ is pseudo-Einstein with constant scalar curvature and
torsion free. This proves the first part.

To prove the second part, we now assume that $\phi$ is not constant. First
observe that $\theta$ and $\widetilde{\theta}$ play symmetric roles in the
statement. Therefore it suffices to do it for $\theta$. As $R_{\alpha
\overline{\beta}}=\frac{m+1}{2}\delta_{\alpha\beta}$ and $A_{\alpha\beta}=0$,
\ it is easy to check by Proposition \ref{Rica} in the Appendix that the
adapted Riemannian metric $g_{\theta}$ is Einstein: $Ric\left(  g_{\theta
}\right)  =\frac{m}{2}g_{\theta}$. Since the Ricci curvature is positive, $M$
has a finite fundamental group. We can work on its universal covering
$\widetilde{M}$, which is still a closed pseudohermitian manifold. For
simplicity we will use the same letter for both the object on $M$ and its
pullback on $\widetilde{M}$. Let $u=\log\phi$. Then%
\begin{align*}
u_{\alpha,\beta}  &  =-u_{\alpha}u_{\beta},\\
u_{\alpha,\overline{\beta}}  &  =\frac{1}{2}\left(  \frac{1}{2}e^{-u}-\frac
{1}{2}+\left\vert \partial u\right\vert ^{2}+\sqrt{-1}u_{0}\right)
\delta_{\alpha\overline{\beta}},\\
u_{0,\alpha}  &  =-\frac{1}{2}u_{0}u_{\alpha}+\frac{\sqrt{-1}}{2}\left(
\frac{1}{2}+\frac{1}{2}e^{-u}+\left\vert \partial u\right\vert ^{2}\right)
u_{\alpha}.
\end{align*}
We now claim that $u$ is CR pluriharmonic. The argument is the same as in
\cite{JL2}. Indeed, when $m\geq2$ \ this follows from the 2nd equation. When
$m=1$, differentiating the 2nd equation and simplifying using all three yields%
\[
u_{\overline{1},11}=\frac{1}{2}\left(  -\frac{1}{2}e^{-u}u_{1}+u_{1,1}%
u_{\overline{1}}+u_{1}u_{\overline{1},1}-\sqrt{-1}u_{0,1}\right)  =0.
\]
As $A_{11}=0$, it follows that $u$ is CR pluriharmonic by \cite[Proposition
3.4]{Lee}. As $\widetilde{M}$ is simply connected, $u$ is the real part of a
CR holomorphic function $u+\sqrt{-1}v$:
\[
v_{\alpha}=-\sqrt{-1}u_{\alpha},v_{\overline{\beta}}=\sqrt{-1}u_{\overline
{\beta}}.
\]
We also have
\begin{align*}
\sqrt{-1}v_{0}\delta_{\alpha\beta}  &  =v_{\alpha,\overline{\beta}%
}-v_{\overline{\beta},\alpha}\\
&  =-\sqrt{-1}u_{\alpha,\overline{\beta}}-\sqrt{-1}u_{\overline{\beta},\alpha
}\\
&  =-2\sqrt{-1}u_{\alpha,\overline{\beta}}-u_{0}\delta_{\alpha\beta}\\
&  =-\sqrt{-1}\left(  \frac{1}{2}e^{-u}-\frac{1}{2}+\left\vert \partial
u\right\vert ^{2}\right)  \delta_{\alpha\overline{\beta}}%
\end{align*}
Thus
\[
v_{0}=-\left(  \frac{1}{2}e^{-u}-\frac{1}{2}+\left\vert \partial u\right\vert
^{2}\right)  .
\]
With this we can rewrite the equations satisfied by $u$ as%
\begin{align*}
u_{\alpha,\beta}  &  =-u_{\alpha}u_{\beta},\\
u_{\alpha,\overline{\beta}}  &  =\frac{1}{2}\left(  -v_{0}+\sqrt{-1}%
u_{0}\right)  \delta_{\alpha\overline{\beta}},\\
u_{0,\alpha}  &  =-\frac{1}{2}u_{0}u_{\alpha}+\frac{\sqrt{-1}}{2}\left(
1-v_{0}\right)  u_{\alpha}%
\end{align*}
Let $f=\exp u/2\cos v/2-c$, where $c$ is a constant such that $\int
_{\widetilde{M}}f=0$.

\begin{proposition}
\label{crh}We have%
\begin{align*}
f_{\alpha,\beta}  &  =0,\\
f_{\alpha,\overline{\beta}}  &  =\frac{1}{2}\left[  -\left(  e^{u/2}\sin
v/2\right)  _{0}+\sqrt{-1}f_{0}\right]  \delta_{\alpha\overline{\beta}},\\
f_{0,\alpha}  &  =\frac{\sqrt{-1}}{2}f_{\alpha},\\
f_{0,0}  &  =-\frac{1}{2}\left(  e^{u/2}\sin v/2\right)  _{0}.
\end{align*}

\end{proposition}

\begin{proof}
These formula are proved by direct calculations. For example, to prove the 3rd
one we first observe as $v_{\alpha}=-\sqrt{-1}u_{\alpha}$ and $\theta$ is
torsion-free
\[
v_{0,\alpha}=v_{\alpha,0}=-\sqrt{-1}u_{\alpha,0}.
\]
Then we compute using the 3rd equation for $u$%
\begin{align*}
f_{0,\alpha}  &  =\frac{1}{2}e^{u/2}\left[  \left(  u_{0,\alpha}-\frac{1}%
{2}v_{0}v_{\alpha}+\frac{1}{2}u_{0}u_{\alpha}\right)  \cos\frac{v}{2}-\left(
v_{0,\alpha}+\frac{1}{2}u_{0}v_{\alpha}+\frac{1}{2}v_{0}u_{\alpha}\right)
\sin\frac{v}{2}\right] \\
&  =\frac{1}{2}e^{u/2}\left[  \left(  u_{0,\alpha}+\frac{\sqrt{-1}}{2}%
v_{0}u_{\alpha}+\frac{1}{2}u_{0}u_{\alpha}\right)  \cos\frac{v}{2}-\left(
-\sqrt{-1}u_{0,\alpha}-\frac{\sqrt{-1}}{2}u_{0}u_{\alpha}+\frac{1}{2}%
v_{0}u_{\alpha}\right)  \sin\frac{v}{2}\right] \\
&  =\frac{1}{2}e^{u/2}\left(  \frac{\sqrt{-1}}{2}u_{\alpha}\cos\frac{v}%
{2}+\frac{1}{2}u_{\alpha}\sin\frac{v}{2}\right) \\
&  =\frac{\sqrt{-1}}{2}e^{u/2}\left(  \frac{1}{2}u_{\alpha}\cos\frac{v}%
{2}-\frac{1}{2}v_{\alpha}\sin\frac{v}{2}\right) \\
&  =\frac{\sqrt{-1}}{2}f_{\alpha}.
\end{align*}
The 1st and 2nd formulas can be proved similarly.

To prove the last identity, we differentiate the 3rd one%
\begin{align*}
\frac{\sqrt{-1}}{2}f_{\alpha,\overline{\beta}}  &  =f_{0,\alpha\overline
{\beta}}\\
&  =f_{0,\overline{\beta}\alpha}+\sqrt{-1}f_{0,0}\delta_{\alpha\overline
{\beta}}\\
&  =\overline{f_{0,\beta\overline{\alpha}}}+\sqrt{-1}f_{0,0}\delta
_{\alpha\overline{\beta}}\\
&  =-\frac{\sqrt{-1}}{2}\overline{f_{\beta,\overline{\alpha}}}+\sqrt
{-1}f_{0,0}\delta_{\alpha\overline{\beta}}%
\end{align*}
Using the 2nd identity we obtain%
\[
f_{0,0}=-\frac{1}{2}\left(  e^{u/2}\sin\frac{v}{2}\right)  _{0}%
\]

\end{proof}

Let $D^{2}f$ denote the Hessian of $f$ w.r.t. the adapted Riemannian metric
$g_{\theta}$. By Proposition \ref{Hess} in Appendix, we obtain from
Proposition \ref{crh}%
\begin{equation}
D^{2}f=-\frac{1}{2}\left(  e^{u/2}\sin\frac{v}{2}\right)  _{0}g_{\theta}.
\label{hf}%
\end{equation}

We pause to prove a simple lemma in Riemannian geometry.

\begin{proposition}
Let $\left(  \Sigma^{n},g\right)  $ be a closed Riemannian manifold s.t.
$Ric\left(  g\right)  =\left(  n-1\right)  c^{2}g$ with $c>0$ a constant.
Suppose $u\in C^{\infty}\left(  \Sigma\right)  $ is a nonzero function s.t.
$\int_{\Sigma}u=0$ and
\begin{equation}
D^{2}u=-\chi g \label{hfchi}%
\end{equation}
for some $\chi\in C^{\infty}\left(  \Sigma\right)  $. Then $\left(  \Sigma
^{n},g\right)  $ is isometric to the unit sphere $\mathbb{S}^{n}$ in the
Euclidean space $\mathbb{R}^{n+1}$ with the metric $\frac{1}{c}g_{0}$ and $u$
corresponds to a linear function on $\mathbb{S}^{n}$, where $g_{0}$ is the
canonical metric on $\mathbb{S}^{n}$.
\end{proposition}

\begin{proof}
Taking trace of (\ref{hfchi}) yields $\Delta u=-n\chi$. Working with a local
orthonormal frame we differentiate (\ref{hfchi})%
\begin{align*}
-\chi_{i}  &  =u_{ji,j}\\
&  =u_{jj,i}+R_{ijlj}u_{l}\\
&  =\left(  \Delta u\right)  _{i}+R_{il}u_{l}\\
&  =-n\chi_{i}+\left(  n-1\right)  c^{2}u_{i}.
\end{align*}
Thus $\chi_{i}-c^{2}u_{i}=0$ or $\chi-c^{2}u$ is constant. Since $\int
_{\Sigma}u=0$, we have $\chi=c^{2}u$. Therefore $D^{2}u=-c^{2}ug$. The
proposition then follows from the classic Obata theorem \cite{O1}.
\end{proof}

Since $\left(  \widetilde{M},g_{\theta}\right)  $ is Einstein with $Ric\left(
g_{\theta}\right)  =\frac{m}{2}g$ and $f$ satisfies (\ref{hf}), applying the
above Proposition we conclude that $\left(  \widetilde{M},g_{\theta}\right)  $
is isometric to $\left(  \mathbb{S}^{2m+1},4g_{0}\right)  $ and $\ f$
corresponds to a linear function on $\mathbb{S}^{2m+1}$. By an argument in
\cite{LW}$\left(  \widetilde{M},\theta\right)  $ is in fact CR isometric to
$\left(  \mathbb{S}^{2m+1},\theta_{c}\right)  $. For completeness, we repeat
the proof here. Without loss of generality, we can take $(\widetilde
{M},g_{\theta})$ to be $\left(  \mathbb{S}^{2m+1},4g_{0}\right)  $. Then
$\theta$ is a pseudohermitian structure on $\mathbb{S}^{2m+1}$ whose adapted
metric is $4g_{0}$ and the associated Tanaka-Webster connection is
torsion-free. It is a well known fact that the Reeb vector field $T$ is then a
Killing vector field for $g_{0}$ . Therefore there exists a skew-symmetric
matrix $A$ such that for all $X\in\mathbb{S}^{2m+1},T(X)=AX$, here we use the
obvious identification between $z=(z_{1},\ldots,z_{m+1})\in\mathbb{C}^{m+1}$
and $X=(x_{1},y_{1},\ldots,x_{m+1},y_{m+1})\in\mathbb{R}^{2m+2}$. Changing
coordinates by an orthogonal transformation we can assume that $A$ is of the
following form
\[
A=\left[
\begin{array}
[c]{ccc}%
\begin{array}
[c]{cc}%
0 & a_{1}\\
a_{1} & 0
\end{array}
&  & \\
& \ddots & \\
&  &
\begin{array}
[c]{cc}%
0 & a_{m+1}\\
a_{m+1} & 0
\end{array}
\end{array}
\right]
\]
where $a_{i}\geq0$. Therefore
\[
T=\sum_{i}a_{i}\left(  y_{i}\frac{\partial}{\partial x_{i}}-x_{i}%
\frac{\partial}{\partial y_{i}}\right)
\]
Since $T$ is of unit length we must have
\[
4\sum_{i}a_{i}^{2}(x_{i}^{2}+y_{i}^{2})=1
\]
on $\mathbb{S}^{2m+1}$. Therefore all the $a_{i}$'s are equal to $1/2$. It
follows that
\[
\theta=g_{0}(T,\cdot)=2\sqrt{-1}\overline{\partial}|z|^{2}.
\]
Therefore $\left(  \widetilde{M},\theta\right)  $ is CR isometric to $\left(
\mathbb{S}^{2m+1},\theta_{c}\right)  $.

Take $\left(  \widetilde{M},\theta\right)  $ to be $\left(  \mathbb{S}%
^{2m+1},\theta_{c}\right)  $. Then there exists a unit $\xi\in\mathbb{C}%
^{m+1}$ and $a>0$ s.t.
\[
f\left(  z\right)  =a\operatorname{Re}z\cdot\overline{\xi}.
\]

Now $M=\mathbb{S}^{2m+1}/\Gamma$, where $\Gamma\subset U\left(  m+1\right)  $
is a finite group acting on $\mathbb{S}^{2m+1}$ freely. Since $f$ must be
invariant under $\Gamma$, it is easy to see that $\Gamma$ must be trivial.
Finally, we have%
\begin{align*}
\operatorname{Re}e^{u/2+\sqrt{-1}v/2}  &  =\exp u/2\cos v/2\\
&  =c+a\operatorname{Re}z\cdot\overline{\xi}\\
&  =\operatorname{Re}\left(  c+az\cdot\overline{\xi}\right)  .
\end{align*}
Thus $e^{u/2+\sqrt{-1}v/2}=\lambda+az\cdot\overline{\xi}$ with $\lambda
=c+\sqrt{-1}c^{\prime}$ for some $c^{\prime}\in\mathbb{R}$. Then
\[
\phi=e^{u}=\left\vert \lambda+az\cdot\overline{\xi}\right\vert ^{2}.
\]
This finishes the proof.

\section{Appendix}

In this appendix, we collect some of the formulas in CR geometry used in the
proof of the main theorem.

Let $\left(  M^{2m+1},\theta\right)  $ be pseudohermitian manifold and
$\nabla$ the Tanaka-Webster connection. Let $\widehat{\nabla}$ be the
Levi-Civita connection of the adapted Riemannian metric $g_{\theta}$. The
following two propositions can be found for example in \cite{DT} in equivalent forms.

\begin{proposition}
\label{connd}We have%
\begin{align*}
\widehat{\nabla}_{X}Y  &  =\nabla_{X}Y+\theta\left(  Y\right)  AX+\frac{1}%
{2}\left(  \theta\left(  Y\right)  \phi X+\theta\left(  X\right)  \phi
Y\right) \\
&  -\left[  \left\langle AX,Y\right\rangle +\frac{1}{2}\omega\left(
X,Y\right)  \right]  T.
\end{align*}

\end{proposition}

With this formula, one can compare the curvature tensor $R$ of $\nabla$ and
the curvature tensor $\widehat{R}$ of $\widehat{\nabla}$.

\begin{proposition}
\label{cuv}Suppose $X,Y$ are horizontal vector fields, then%
\begin{align*}
\widehat{R}\left(  X,Y,X,Y\right)   &  =R\left(  X,Y,X,Y\right)  -\frac{3}%
{4}\left\langle JX,Y\right\rangle ^{2}+\left\langle AX,Y\right\rangle
^{2}-\left\langle AX,X\right\rangle \left\langle AY,Y\right\rangle ,\\
\widehat{R}\left(  X,T,Y,T\right)   &  =-\left\langle \nabla_{T}%
AX,Y\right\rangle -\left\langle AX,AY\right\rangle +\left\langle
AX,JY\right\rangle +\frac{1}{4}\left\langle X,Y\right\rangle ,\\
\widehat{R}\left(  X,Y,Z,T\right)   &  =\left\langle \nabla_{X}%
AY,Z\right\rangle -\left\langle \nabla_{Y}AX,Z\right\rangle .
\end{align*}

\end{proposition}

Using Proposition \ref{connd}, it is easy to check by direct calculation the following

\begin{proposition}
\label{Hess}Let $u\in C^{\infty}\left(  M\right)  $ and $D^{2}u$ be its
Riemannian Hessian w.r.t. $g_{\theta}$. We have the following formulas%
\begin{align*}
&  D^{2}u\left(  T,T\right)  =u_{0,0},\\
&  D^{2}u\left(  T,T_{\alpha}\right)  =u_{\alpha,0}-\frac{\sqrt{-1}}%
{2}u_{\alpha},\\
&  D^{2}u\left(  T_{\alpha},T_{\beta}\right)  =u_{a,\beta}+A_{\alpha\beta
}u_{0}\\
&  D^{2}u\left(  T_{\alpha},T_{\overline{\beta}}\right)  =u_{\alpha
,\overline{\beta}}-\frac{\sqrt{-1}}{2}\delta_{\alpha\beta}u_{0}%
\end{align*}

\end{proposition}

\bigskip

Taking trace using Proposition \ref{cuv} yields the Ricci curvature $Ric$ of
$g_{\theta}$ in terms of the pseudohermitian Ricci tensor $R_{\alpha
\overline{\beta}}$ and the torsion $A$.

\begin{proposition}
\label{Rica}Suppose $X=2\operatorname{Re}\sum_{\alpha=1}^{m}c_{\alpha
}T_{\alpha}$. We have%
\begin{align*}
Ric\left(  X,X\right)   &  =2R_{\alpha\overline{\beta}}c_{\alpha}%
\overline{c_{\beta}}+\sqrt{-1}\left(  m-1\right)  \left(  A_{\alpha\beta
}c_{\alpha}c_{\beta}-A_{\overline{\alpha}\overline{\beta}}\overline{c_{\alpha
}}\overline{c_{\beta}}\right)  \\
&  -\frac{1}{2}\left\vert X\right\vert ^{2}-\left\langle \nabla_{T}%
AX,X\right\rangle +\left\langle AX,JX\right\rangle ,\\
Ric\left(  X,T\right)   &  =2\left\langle X,\operatorname{Re}A_{\alpha
\beta,\overline{\alpha}}T_{\overline{\beta}}\right\rangle ,\\
Ric\left(  T,T\right)   &  =\frac{m}{2}-\left\vert A\right\vert ^{2}.
\end{align*}

\end{proposition}

\end{document}